\documentclass[11pt]{amsart}

\usepackage[T1]{fontenc}
\usepackage{fullpage}
\usepackage[utf8]{inputenc}
\usepackage{lmodern}
\usepackage{microtype}
\usepackage{amssymb,mathtools}
\usepackage[hidelinks]{hyperref}

\hypersetup{
  pdftitle={A B\'ezout domain that is not an elementary divisor domain},
  pdfauthor={Christian H\"agg and Anders M\"ortberg},
  pdfsubject={B\'ezout domains, elementary divisor domains, and Smith normal form}
}
\numberwithin{equation}{section}

\newtheorem{theorem}{Theorem}[section]
\newtheorem{lemma}[theorem]{Lemma}
\newtheorem{proposition}[theorem]{Proposition}
\theoremstyle{definition}
\newtheorem{definition}[theorem]{Definition}
\newtheorem{construction}[theorem]{Construction}
\newtheorem{notation}[theorem]{Notation}

\newcommand{\Q}{\mathbb{Q}}
\newcommand{\R}{\mathbb{R}}
\newcommand{\N}{\mathbb{N}}
\newcommand{\Z}{\mathbb{Z}}
\newcommand{\fm}{\mathfrak{m}}
\newcommand{\fn}{\mathfrak{n}}
\newcommand{\fp}{\mathfrak{p}}
\newcommand{\fq}{\mathfrak{q}}
\newcommand{\fa}{\mathfrak{a}}
\newcommand{\cF}{\mathcal{F}}
\newcommand{\cR}{\mathcal{R}}
\newcommand{\fb}{\mathfrak{b}}
\DeclareMathOperator{\Spec}{Spec}
\DeclareMathOperator{\GL}{GL}
\DeclareMathOperator{\diag}{diag}
\DeclareMathOperator{\inv}{inv}
\DeclareMathOperator{\maxinv}{maxinv}
\DeclareMathOperator{\ord}{ord}

\title{A B\'ezout domain that is not an elementary divisor domain}
\author[Christian H\"agg]{Christian H\"agg}
\address{Department of Mathematics, Stockholm University}
\email{hagg@math.su.se}
\author[Anders M\"ortberg]{Anders M\"ortberg}
\address{Department of Mathematics, Stockholm University}
\email{anders.mortberg@math.su.se}
\date{\today}

\begin{document}

\begin{abstract}
We settle in the negative the longstanding question whether every B\'ezout domain is an elementary divisor domain by constructing a B\'ezout domain over which an explicit $2\times2$ matrix has no Smith normal form. The obstruction is topological and is detected by the M\"obius line bundle.
\end{abstract}

\maketitle

\noindent\emph{This is an early draft. The complete Lean formalization is available at \url{https://github.com/Zelaron/bezout-counterexample-lean}.}

\bigskip

\section{Introduction}

The Smith normal form over a principal ideal domain is a fundamental structure theorem: every matrix can be diagonalized by invertible row and column operations, with diagonal entries forming a divisibility chain. In 1943, Helmer asked whether this conclusion continues to hold over every B\'ezout domain \cite{Helmer}.\footnote{Helmer uses the term ``Prüfer ring'' for what we here call ``B\'ezout domain''.}  We answer this question in the negative.

Throughout, rings are commutative with identity, and ring homomorphisms preserve the identity. A \emph{B\'ezout domain} is an integral domain in which every finitely generated ideal is principal; equivalently, every two-generated ideal is principal. For a ring $A$, two $m\times n$ matrices $F$ and $G$ over $A$ are \emph{equivalent} if $PFQ=G$ for some $P\in\GL_m(A)$ and $Q\in\GL_n(A)$. A \emph{Smith normal form} of $F$ is an equivalent diagonal matrix whose nonzero diagonal entries occur first and satisfy $d_1\mid d_2\mid\cdots\mid d_r$. A domain is an \emph{elementary divisor domain} if every matrix over it has a Smith normal form. Every elementary divisor domain is B\'ezout, as follows by applying the definition to $1\times2$ matrices. The problem is whether the converse holds.

Helmer proved the elementary divisor theorem for adequate domains \cite[Theorem~3]{Helmer}. Kaplansky subsequently introduced elementary divisor rings and showed that, over a domain, the problem is already detected by $2\times2$ matrices \cite[Theorem~5.1 and the subsequent remark]{Kaplansky}. Gillman and Henriksen constructed a B\'ezout ring with zero divisors that is not an elementary divisor ring \cite[Example~4.11]{GH}, but the domain case remained open.

Our construction starts from
\begin{equation}\label{eq:initial-data}
  A_0=\Q[x,y],\qquad
  \Delta=x^2+y^2-1,\qquad
  M=\begin{pmatrix}1+x&y\\ y&1-x\end{pmatrix},
\end{equation}
where $x$ and $y$ are algebraically independent over $\Q$.

\begin{theorem}\label{thm:main}
There exists a B\'ezout domain that is not an elementary divisor domain.
\end{theorem}

The proof constructs an ascending sequence of smooth finitely generated factorial $\Q$-domains $A_n$ together with compact sets $K_n\subset\Spec(A_n)(\R)$. Every transition map $K_{n+1}\to K_n$ is a monotone surjection, so the pullback of the M\"obius line bundle on the real unit circle remains nonorientable. At the same time, every two-generated ideal occurring at some stage becomes principal at a later stage. The union $R=\bigcup_n A_n$ is therefore a B\'ezout domain. An equivalence of $M$ to Smith normal form over $R$ would already be defined at one finite stage and would yield a nowhere-zero section of the corresponding nonorientable line bundle.

The algebraic point is to principalize a two-generated ideal without destroying the topological obstruction. We use the weighted invariant of Abramovich, Temkin, and W{\l}odarczyk \cite{ATW}, in the presentation of Brais \cite{Brais}. For a center of codimension at least two, we form its extended Rees algebra and apply Jouanolou's device. The resulting smooth factorial affine ring factors the ideal by a prime power and lowers the invariant or removes one component of its maximal locus. Divisorial centers are handled by division. On real points, the torsor step admits a compact lift with connected fibers.

\bigskip

\paragraph{\textbf{Outline}} Section~\ref{sec:topology} gives the topological obstruction, Section~\ref{sec:tools} records the algebraic tools, Section~\ref{sec:principalization} proves the principalization extension used in the construction, and Section~\ref{sec:construction} constructs the B\'ezout domain and proves Theorem~\ref{thm:main}.

\bigskip

\paragraph{\textbf{Use of AI}} OpenAI's GPT-6 Astra (via Codex) was used in preparing this article and its Lean formalization. Anthropic's Claude Opus 5.5 was used to first (unsuccessfully) search for a counterexample and then to check the counterexample produced by GPT-6 Astra.

\bigskip

\section{The topological obstruction}\label{sec:topology}

For elements $a_1,\ldots,a_s$ of a ring, $(a_1,\ldots,a_s)$ denotes the ideal they generate. Direct calculation from \eqref{eq:initial-data} gives
\begin{equation}\label{eq:matrix-identities}
  \det M=-\Delta,\qquad
  (1+x,y,1-x)=A_0,\qquad
  M^2-2M=\Delta I_2.
\end{equation}
The middle identity follows from $(1+x)+(1-x)=2\in\Q^\times$.

\subsection{Real points}\label{subsec:real-points}
For a ring $A$, let $\Spec(A)(\R)$ be the set of ring homomorphisms $z:A\to\R$, with the coarsest topology for which every evaluation map $z\mapsto z(a)$ is continuous. A homomorphism $f:A\to B$ induces the continuous map $f^*:\Spec(B)(\R)\to\Spec(A)(\R)$, $z\mapsto z\circ f$. If $A=\Q[a_1,\ldots,a_m]$ is finitely generated, then $z\mapsto(z(a_1),\ldots,z(a_m))$ is a homeomorphism onto the real zero set in $\R^m$ of the relations among the $a_i$. Thus $\Spec(A)(\R)$ is the usual space of real points with its Euclidean topology. In particular, it is Hausdorff, and a subset of $\Spec(A)(\R)$ is compact if and only if it is closed and the generators $a_1,\ldots,a_m$ are bounded on it.

A continuous surjection $q:K'\to K$ is \emph{monotone} if every fiber $q^{-1}(z)$ is connected. In our applications, $K'$ and $K$ are compact Hausdorff spaces.

\subsection{The M\"obius line bundle}
Let $K$ be a topological space, and let $X,Y:K\to\R$ be continuous functions satisfying $X^2+Y^2=1$. Write $M(X,Y)$ for the matrix obtained from $M$ by substituting $X$ and $Y$ for $x$ and $y$. For every $z\in K$, the matrix $\frac12M(X(z),Y(z))$ is an idempotent of rank one by \eqref{eq:matrix-identities}. Let $L_{X,Y}(z)\subset\R^2$ be its image. These fibers form a real line subbundle $L_{X,Y}\subset K\times\R^2$. We call $L_{X,Y}$ \emph{orientable} if it admits a continuous nowhere-zero section, that is, a continuous map $v:K\to\R^2$ such that $0\ne v(z)\in L_{X,Y}(z)$ for every $z\in K$. If $q:K'\to K$ is continuous, then the pullback of $L_{X,Y}$ along $q$ is $L_{X\circ q,Y\circ q}$.

Let
\[
  K_0=\{z\in\Spec(A_0)(\R):z(\Delta)=0\},
\]
which is the unit circle via $z\mapsto(z(x),z(y))$, and let $L_0=L_{x,y}$ be the corresponding line bundle. For $(x,y)=(\cos\theta,\sin\theta)$,
\begin{equation}\label{eq:half-angle}
  M=2
  \begin{pmatrix}\cos(\theta/2)\\ \sin(\theta/2)\end{pmatrix}
  \begin{pmatrix}\cos(\theta/2)&\sin(\theta/2)\end{pmatrix}.
\end{equation}
Thus $L_0$ is the M\"obius line bundle. Indeed, a continuous nowhere-zero section, pulled back to $[0,2\pi]$, would have the form
\[
  f(\theta)\bigl(\cos(\theta/2),\sin(\theta/2)\bigr)
\]
for a continuous nowhere-zero function $f$. Equality at the endpoints would imply $f(2\pi)=-f(0)$, contradicting the intermediate value theorem. Hence $L_0$ is nonorientable.

\begin{lemma}\label{lem:orientation-descent}
Let $q:K'\to K$ be a monotone surjection from a compact space to a Hausdorff space, and let $X,Y:K\to\R$ be continuous with $X^2+Y^2=1$. If $L_{X\circ q,Y\circ q}$ is orientable, then $L_{X,Y}$ is orientable.
\end{lemma}

\begin{proof}
Let $v$ be a continuous nowhere-zero section of $L_{X\circ q,Y\circ q}$, and put $u=v/\|v\|$. For $z\in K$, the line $L_{X,Y}(z)$ contains exactly two unit vectors. The continuous map $u$ therefore takes values in a two-point set on the connected fiber $q^{-1}(z)$, so it is constant there. Hence $u=\bar u\circ q$ for a map $\bar u:K\to\R^2$ such that $\bar u(z)$ is a unit vector in $L_{X,Y}(z)$. Because $K'$ is compact and $K$ is Hausdorff, $q$ is a closed map and hence a quotient map. Therefore $\bar u$ is continuous, and it orients $L_{X,Y}$.
\end{proof}

Equivalently, nonorientability pulls back along monotone surjections between compact Hausdorff spaces.

\begin{lemma}\label{lem:monotone-composition}
A composite of monotone surjections between compact Hausdorff spaces is monotone.
\end{lemma}

\begin{proof}
We first note that if $q:K'\to K$ is monotone and $C\subset K$ is compact and connected, then $q^{-1}(C)$ is connected. Otherwise, write $q^{-1}(C)=F\amalg G$ with $F$ and $G$ nonempty and closed. Each fiber of $q$ meets at most one of $F$ and $G$, because it is connected. The compact sets $q(F)$ and $q(G)$ are therefore disjoint nonempty closed subsets whose union is $C$, a contradiction. Applying this observation to every fiber of a second monotone surjection proves the claim.
\end{proof}

\section{Algebraic tools}\label{sec:tools}

A domain is \emph{factorial} if it is a unique factorization domain (UFD). We repeatedly use finitely generated $\Q$-domains that are smooth over $\Q$ and factorial. Such rings are Noetherian.

\subsection{Factoriality}
We use two standard facts. A polynomial ring over a UFD is a UFD \cite[Tag~0BC1]{Stacks}. Nagata's criterion states that if $A$ is a Noetherian domain and $p\in A$ is a prime element such that $A[1/p]$ is a UFD, then $A$ is a UFD \cite[Tag~0AFU]{Stacks}. Consequently, if $A$ is a Noetherian UFD, then so is $A[T,T^{-1}]=A[T][1/T]$. Moreover, every prime element $p$ of $A$ remains prime in $A[T,T^{-1}]$, because $(A/pA)[T,T^{-1}]$ is a domain.

\subsection{Jouanolou rings}
Let $B$ be a ring and let $c=(c_0,\ldots,c_r)\in B^{r+1}$. Put
\[
  q_c=\sum_{i=0}^r c_i\sigma_i-1\in B[\sigma_0,\ldots,\sigma_r].
\]
The \emph{Jouanolou ring} of $c$ is
\[
  J_B(c)=B[\sigma_0,\ldots,\sigma_r]/(q_c).
\]
Let $W=\Spec B\setminus V(c_0,\ldots,c_r)$. Over $W$, the row $(c_0,\ldots,c_r)$ defines a surjection $\mathcal O_W^{\oplus(r+1)}\to\mathcal O_W$ with locally free kernel $E_c$ of rank $r$. The morphism $\Spec J_B(c)\to W$ is an $E_c$-torsor, hence a Zariski-locally trivial affine-space bundle with affine total space. This is an explicit instance of Jouanolou's device \cite[Lemme~1.5]{Jouanolou}. A real point of $J_B(c)$ is a pair $(z,t)$ with $z\in\Spec(B)(\R)$ and $t=(t_0,\ldots,t_r)\in\R^{r+1}$ satisfying $\sum_{i=0}^r t_i z(c_i)=1$.

\begin{lemma}\label{lem:jouanolou}
Let $B$ be a ring and $c\in B^{r+1}$.
\begin{enumerate}
\item For every homomorphism $B\to B'$, one has $J_B(c)\otimes_B B'\cong J_{B'}(c')$, where $c'$ is the image of $c$. In particular, $J_B(c)/\fa J_B(c)\cong J_{B/\fa}(\bar c)$ for every ideal $\fa\subset B$, and the formation of $J_B(c)$ commutes with localization of $B$.
\item For every $j\in\{0,\ldots,r\}$, eliminating $\sigma_j$ gives $J_B(c)[1/c_j]\cong B[1/c_j][\sigma_i\mid i\ne j]$.
\item If $B$ is a domain and $c_j\ne0$ for some $j$, then $J_B(c)$ is a domain and $B\to J_B(c)$ is injective.
\item $J_B(c)$ is formally smooth over $B$. If $B$ is a smooth finitely generated $\Q$-algebra, then so is $J_B(c)$.
\item If $B$ is a Noetherian UFD, $c_j$ is a prime element of $B$, and $c_j\nmid c_m$ for some $m$, then $J_B(c)$ is a UFD.
\end{enumerate}
\end{lemma}

\begin{proof}
Assertion (1) follows from the right exactness of the tensor product, and (2) from solving $q_c=0$ for $\sigma_j$.

For (3), the constant term of $q_c$ is $-1$. For every $b\in B$, the image of $q_c$ in $(B/bB)[\sigma_0,\ldots,\sigma_r]$ is therefore a unit in the corresponding formal power series ring, and hence a nonzerodivisor in the polynomial ring. If $B$ is a domain and $0\ne b\in B$, it follows that multiplication by $b$ is injective on $J_B(c)$: from $bf=gq_c$ in $B[\sigma_0,\ldots,\sigma_r]$, reduction modulo $b$ gives $g\in bB[\sigma_0,\ldots,\sigma_r]$, and cancellation of $b$ gives $f\in(q_c)$. Applying this to the powers of $c_j$ shows that $J_B(c)\to J_B(c)[1/c_j]$ is injective. By (2), the target is a polynomial ring over the domain $B[1/c_j]$ and contains $B$.

For (4), put $P=B[\sigma_0,\ldots,\sigma_r]$. Since $q_c$ has constant term $-1$, it is a unit in $B[[\sigma_0,\ldots,\sigma_r]]$ and hence a nonzerodivisor in $P$. The $B$-derivation $\mathcal E=\sum_{i=0}^r\sigma_i\,\partial/\partial\sigma_i$ of $P$ satisfies $\mathcal E(q_c)=q_c+1$. It therefore induces a map $\Omega_{P/B}\otimes_P J_B(c)\to J_B(c)$ that sends $dq_c$ to $1$. Thus the conormal sequence
\[
  0\longrightarrow(q_c)/(q_c^2)\longrightarrow
  \Omega_{P/B}\otimes_P J_B(c)\longrightarrow\Omega_{J_B(c)/B}
  \longrightarrow0
\]
is split exact, so $J_B(c)$ is formally smooth over $B$ by the split-conormal criterion \cite[Tag~031J]{Stacks}. If $B$ is smooth and finitely generated over $\Q$, then $J_B(c)$ is formally smooth and of finite presentation over $\Q$, hence smooth.

For (5), the ring $J_B(c)[1/c_j]$ is a UFD by (2) and the facts recalled above. By (1), $J_B(c)/c_jJ_B(c)\cong J_{B/c_jB}(\bar c)$. This is a domain by (3), because $B/c_jB$ is a domain and $\bar c_m\ne0$. By (3) again, $c_j$ is nonzero in $J_B(c)$, so $c_j$ is a prime element of the Noetherian domain $J_B(c)$. Nagata's criterion shows that $J_B(c)$ is a UFD.
\end{proof}

\subsection{Marked centers and the weighted invariant}
Let $A$ be a smooth finitely generated $\Q$-domain and let $\fm$ be a maximal ideal. Then $A_\fm$ is a regular local ring of dimension $n=\dim A$, and its residue field is a finite extension of $\Q$. If $x=(x_1,\ldots,x_n)$ is a regular system of parameters of $A_\fm$, then $dx_1,\ldots,dx_n$ is a basis of $\Omega_{A_\fm/\Q}$, because $\fm/\fm^2\to\Omega_{A/\Q}\otimes\kappa(\fm)$ is an isomorphism in characteristic zero. Let $\partial_1,\ldots,\partial_n$ be the dual derivations. We use the convention $\N=\{0,1,2,\ldots\}$.

\begin{definition}\label{def:marked-centre}
A \emph{marked center} at $\fm$ is a pair $J=(x,e)$ consisting of a regular system of parameters $x$ of $A_\fm$ and rational weights $e_1\ge e_2\ge\cdots\ge e_n\ge0$. For $t\in\Q$, its \emph{weighted ideal of level $t$} is
\[
  F_t(J)=\Bigl(x^\alpha:\;
  \begin{aligned}[t]
    &\alpha\in\N^n,\quad \alpha_i=0\text{ if }e_i=0,\\[-2pt]
    &\sum_{i=1}^n e_i\alpha_i\ge t
  \end{aligned}
  \Bigr)A_\fm.
\]
If $I\subseteq A$ is an ideal contained in $\fm$, then $J$ is \emph{admissible} for $I$ if $IA_\fm\subseteq F_1(J)$.
\end{definition}

In the notation of \cite{ATW,Brais}, the marked center $J$ is written $(x_1^{a_1},\ldots,x_k^{a_k})$, where $a_i=1/e_i$ and $k$ is the number of nonzero weights; the condition denoted there by $v_J(I)\ge1$ is precisely admissibility as defined above. We regard every weight vector as an element of $\Q^{\N}$ by appending zeros. Let $\preceq$ be the opposite of the lexicographic order:
\[
  e\preceq e'\quad\Longleftrightarrow\quad e'\le_{\mathrm{lex}} e.
\]
Write $e\prec e'$ when $e\preceq e'$ and $e\ne e'$. This is the order obtained from the lexicographic order in \cite{ATW,Brais} after replacing the entries $a_i$ by the weights $e_i=1/a_i$; with this convention, worse singularities have larger invariants. For example, if $I=(y^2-x^3)\subset\Q[x,y]$ and $\fm=(x,y)$, then the marked center with parameters $(y,x)$ and weights $(\frac12,\frac13)$ is admissible, and $\inv_\fm(I)=(\frac12,\frac13,0,\ldots)$ corresponds to the invariant $(2,3)$ in the notation of \cite{ATW}.

The formulation in \cite[Sections~3--4]{Brais} assumes pure positive codimension. For arbitrary nonzero ideals, we use the corresponding results in \cite[Sections~5--6]{ATW}. All theorem numbers for \cite{ATW} refer to the published version; the numbering in arXiv:1906.07106v3 is different. For a maximal ideal $\fm\supseteq I$, write
\[
  \ord_\fm(I)=\max\{b\in\N:IA_\fm\subseteq\fm^bA_\fm\}.
\]

\begin{theorem}\label{thm:invariant}
Let $A$ be a smooth finitely generated $\Q$-domain and let $I\subset A$ be a nonzero ideal.
\begin{enumerate}
\item \emph{(Invariant and maximal center; \cite[Theorem~5.3.1]{ATW}, \cite[Theorem~3.1.1]{Brais}.)} For every maximal ideal $\fm\supseteq I$ there is an admissible marked center for $I$ at $\fm$ whose weight vector is lexicographically smallest among all admissible marked centers at $\fm$, and all admissible marked centers with this weight vector have the same weighted ideals. We call this weight vector the \emph{invariant} $\inv_\fm(I)$, and the common weighted ideals $F_t(\fm)$ the \emph{maximal center} of $I$ at $\fm$. All weights of $\inv_\fm(I)$ are at most $1$, and its first weight is $1/\ord_\fm(I)$ \cite[Section~5.1]{ATW}. In particular, the first weight has the form $1/a$ for a positive integer $a$, and if $b\ge1$ and $IA_\fm\subseteq\fm^bA_\fm$, then it is at most $1/b$.
\item \emph{(Polynomial-localization invariance; a special case of \cite[Theorem~5.1.3(3)]{ATW} and \cite[Proposition~3.5.1]{Brais}.)} Let $A\to B$ be a homomorphism to a smooth finitely generated $\Q$-domain, let $\fn\supseteq IB$ be a maximal ideal of $B$, and put $\fm=\fn\cap A$. Suppose that, for some $r\ge0$ and prime ideal $\fq$ of $A_\fm[t_1,\ldots,t_r]$, there is an $A$-algebra isomorphism
\[
  B_\fn\cong A_\fm[t_1,\ldots,t_r]_\fq.
\]
Then $\fm$ is maximal and $\inv_\fn(IB)=\inv_\fm(I)$. This is the special case of smooth invariance used below.
\item \emph{(The maximal locus; \cite[Theorems~5.1.3(2) and~6.1.1]{ATW}, \cite[Propositions~4.1.1 and~4.1.2]{Brais}.)} Assume $I\ne A$. Among the invariants $\inv_\fm(I)$, $\fm\supseteq I$ maximal, there is a largest one $\maxinv(I)$ for $\preceq$. There are prime ideals $\fp_1,\ldots,\fp_r$ with $\fp_i+\fp_j=A$ for $i\ne j$ such that $\inv_\fm(I)=\maxinv(I)$ exactly when $\fm$ contains some $\fp_i$. If $k$ is the number of nonzero entries of $\maxinv(I)$ and $\fm\supseteq\fp_i$, there is a regular system of parameters $x$ of $A_\fm$ with $\fp_iA_\fm=(x_1,\ldots,x_k)$ such that $(x,\maxinv(I))$ is admissible for $I$ at $\fm$; by (1), its weighted ideals are the $F_t(\fm)$. We call $V(\fp_1),\ldots,V(\fp_r)$ the \emph{components} of the maximal locus.
\item \emph{(The global center; \cite[Proposition~4.1.1 and Definition~4.1.4]{Brais}, together with affine gluing.)} Let $\fp$ be one of the $\fp_i$. For the center obtained by retaining the component $V(\fp)$ of the global associated center, put
\[
  F_t=\{f\in A:\ f\in F_t(\fm)\text{ for every maximal ideal }\fm\supseteq\fp\}
  \qquad(t\in\Q).
\]
Indeed, restrict Brais's global associated center to an open neighborhood of $V(\fp)$ disjoint from the other components of the maximal locus and, level by level, glue its ideals to the unit ideal on $\Spec(A)\setminus V(\fp)$. The displayed intersection is the ideal of global sections of the resulting level-$t$ ideal sheaf. Hence $F_tA_\fm=F_t(\fm)$ for every maximal ideal $\fm\supseteq\fp$, and $F_tA_\fm=A_\fm$ for every maximal ideal $\fm\not\supseteq\fp$.
\item \emph{(Well-ordering; \cite[Section~5.1]{ATW}, \cite[Proposition~3.3.3]{Brais}.)} There is a set $\Gamma$ of weight vectors, independent of $A$ and $I$, that contains all invariants and such that, for every $N$, its elements with at most $N$ nonzero entries are well-ordered by $\preceq$. This is the reciprocal, zero-padded reformulation of the sets $\Gamma_j$ constructed in \cite[Section~3.3, especially Proposition~3.3.3]{Brais}.
\end{enumerate}
\end{theorem}

In (3), the maximal locus is the set of closed points of $V(\fp_1)\amalg\cdots\amalg V(\fp_r)$; it is smooth of codimension $k$, and each $V(\fp_i)$ is one of its connected components.

\begin{proof}[Justification of the enlargement of $\Gamma$ in (5)]
The set described in (5) may be larger than the set of actual invariants. We explain why it contains the latter and is still well-ordered in bounded length. In the differential construction of \cite[Sections~3.2--3.3]{Brais}, after weights $e_1,\ldots,e_r$ have been chosen, the next positive weight has the form
\[
  e_{r+1}=\frac{1-\sum_{i=1}^r\beta_i e_i}{q},
  \qquad \beta_i\in\N,\quad q\in\N_{>0},\quad
  0<1-\sum_{i=1}^r\beta_i e_i\le1.
\]
Let $D$ be the product of the denominators, in lowest terms, of $e_1,\ldots,e_r$. The numerator is $h/D$ for an integer $1\le h\le D$, and therefore
\[
  e_{r+1}=\frac{h}{Dq}=\frac{D!}{a},
  \qquad a=\frac{D!Dq}{h}\in\N_{>0}.
\]
For the first weight, the empty product gives $D=1$ and $e_1=1/a$.

For a fixed prefix, the possible next positive weights lie in
$\{D!/a:a\in\N_{>0}\}$, which is well-ordered by the reverse of the usual numerical order. Allowing a zero entry, after which every entry is zero, adds a largest element for this order. In a strictly descending sequence for $\preceq$, the first entries must therefore eventually stabilize; once a prefix has stabilized, the next entries must eventually stabilize as well. For sequences of length at most $N$, induction through the $N$ positions rules out an infinite strictly descending sequence. This proves the required well-ordering.
\end{proof}

\begin{notation}\label{not:rees}
Let $A$ be as in Theorem~\ref{thm:invariant}, let $I\subsetneq A$ be a nonzero ideal, let $V(\fp)$ be a component of the maximal locus, and let $e=\maxinv(I)$ have $k$ nonzero entries. Choose an integer $d\ge1$ with $de_i\in\Z$ for all $i$, and index the global center of Theorem~\ref{thm:invariant}(4) by integers: $\cF_j=F_{j/d}$ for $j\in\Z$. Put
\begin{gather*}
  \cR=\bigoplus_{j\in\Z}\cF_jT^j\subset A[T,T^{-1}],\qquad s=T^{-1},\\
  \cR_+=\bigl(\cF_jT^j:j\ge1\bigr)\cR,\qquad
  I_w=\bigl(fT^d:f\in I\bigr)\cR.
\end{gather*}
Thus $\cR$ is the extended Rees algebra of the weighted center along $V(\fp)$, $s$ is the exceptional element, and $I_w$ is the weak transform of $I$. For $1\le j\le d$, let $G_j$ be a finite set of generators of $\cF_j$.
\end{notation}

\begin{lemma}\label{lem:rees}
In the situation of Notation~\ref{not:rees}:
\begin{enumerate}
\item $\cF_j=A$ for $j\le0$, $\cF_1=\fp$, $\fp^j\subseteq\cF_j$ for $j\ge0$, and $I\subseteq\cF_d$. In particular, $I_w\subseteq \cR$, $I\cR=s^dI_w$, and $\fp=\cF_1\subseteq s\cR$.
\item $\cF_j=\sum_{q=1}^d\cF_q\cF_{j-q}$ for $j>d$. Consequently, $\cR=A[s,\ gT^j:\ g\in G_j,\ 1\le j\le d]$, and the elements $gT^j$ with $g\in G_j$ and $1\le j\le d$ generate $\cR_+$.
\item $\cR[1/s]=A[T,T^{-1}]$.
\item Let $\fm\supseteq\fp$ be a maximal ideal, let $x$ be as in Theorem~\ref{thm:invariant}(3), and put $w_i=de_i$ for $i\le k$. Then $u_i\mapsto x_iT^{w_i}$ induces an isomorphism
\[
  A_\fm[s,u_1,\ldots,u_k]/(x_i-s^{w_i}u_i\mid1\le i\le k)\cong \cR\otimes_AA_\fm.
\]
If $\fm\not\supseteq\fp$, then $\cR\otimes_AA_\fm=A_\fm[T,T^{-1}]$. In particular, $\cR$ is a smooth finitely generated $\Q$-domain.
\item $\cR/s\cR\cong\bigoplus_{j\ge0}\cF_j/\cF_{j+1}$ is a domain. Hence $s$ is a prime element of $\cR$.
\end{enumerate}
\end{lemma}

\begin{proof}
By Theorem~\ref{thm:invariant}(4), every assertion about the ideals $\cF_j$ may be checked after localization at maximal ideals. If $\fm\not\supseteq\fp$, then $\cF_jA_\fm=A_\fm$ for every $j$. If $\fm\supseteq\fp$, then $\cF_jA_\fm=F_{j/d}(\fm)$ is generated by the monomials $x^\alpha$ with $\sum_{i=1}^k w_i\alpha_i\ge j$, by Theorem~\ref{thm:invariant}(3). All weights are at most $1$ by Theorem~\ref{thm:invariant}(1), so the weights $w_i$ are integers in $\{1,\ldots,d\}$. This gives $\cF_1A_\fm=(x_1,\ldots,x_k)=\fp A_\fm$ and $\fp^j\subseteq\cF_j$, while $I\subseteq\cF_d$ is admissibility. Since $f=s^d\cdot fT^d$ for $f\in I$ and $g=s\cdot gT$ for $g\in\cF_1$, we have $I\cR=s^dI_w$ and $\cF_1\subseteq s\cR$. This proves (1).

For (2), a monomial of weight at least $j>d$ is the product of some $x_i$, which has weight $w_i\in\{1,\ldots,d\}$, and a monomial of weight at least $j-w_i$. By induction on $j$, every homogeneous element of $\cR$ of positive degree lies in the ideal generated by the elements $gT^j$ and is a polynomial in them with coefficients in $A$. Moreover, $\cF_jT^j=As^{-j}$ for $j\le0$. Assertion (3) holds because $\cF_j=A$ for $j\le0$.

For (4), the formation of $\cR$ commutes with localization, so $\cR\otimes_AA_\fm$ is the extended Rees algebra of the filtration $(F_{j/d}(\fm))_{j\in\Z}$. For a filtration given by monomials in part of a regular system of parameters, the displayed presentation is \cite[Proposition~5.2.2]{QR}; see also \cite[Equation~(2.3)]{Brais}. Its left-hand side is formally smooth over $\Q$ by the criterion \cite[Tag~031J]{Stacks} already used in Lemma~\ref{lem:jouanolou}(4): the derivations $\partial_1,\ldots,\partial_k$ of $A_\fm$, extended to $A_\fm[s,u_1,\ldots,u_k]$ by $\partial_i(s)=\partial_i(u_j)=0$, satisfy $\partial_i(x_j-s^{w_j}u_j)=\delta_{ij}$. Every localization of $\cR$ at a prime ideal is a localization of some $\cR\otimes_AA_\fm$, so it is formally smooth over $\Q$. Since $\cR$ is of finite type over $\Q$ by (2), it is smooth. It is a domain because it is contained in $A[T,T^{-1}]$.

For (5), the identification $\cR/s\cR\cong\bigoplus_j\cF_j/\cF_{j+1}$ follows from $s\cR\cap\cF_jT^j=\cF_{j+1}T^j$. By (4), for every maximal ideal $\fm\supseteq\fp$ the graded ring $\bigoplus_j\cF_j/\cF_{j+1}$ becomes the polynomial ring $(A/\fp)_{\fm/\fp}[u_1,\ldots,u_k]$ after localization at $\fm$. In particular, each $\cF_j/\cF_{j+1}$ is a locally free $A/\fp$-module, hence torsion-free. Now let $f\in\cF_a\setminus\cF_{a+1}$ and $g\in\cF_b\setminus\cF_{b+1}$. Their images in $(\cF_a/\cF_{a+1})_\fm$ and $(\cF_b/\cF_{b+1})_\fm$ are nonzero at every maximal ideal $\fm\supseteq\fp$. Their product in $(A/\fp)_{\fm/\fp}[u_1,\ldots,u_k]$ is therefore nonzero, so $fg\notin\cF_{a+b+1}$. Thus $\bigoplus_j\cF_j/\cF_{j+1}$ is a domain, and it is nonzero because $\cF_1=\fp\ne A$.
\end{proof}

\begin{theorem}[Strict decrease of the invariant]\label{thm:drop}
In the situation of Notation~\ref{not:rees}, let $P$ be a maximal ideal of $\cR$ that contains $I_w+s\cR$ but not $\cR_+$. Then $\inv_P(I_w)\prec\maxinv(I)$.
\end{theorem}

\begin{proof}
We use the affine extended-Rees argument underlying property~(C) in \cite[Section~2.2]{Brais}. Let $\fm=P\cap A$. The weak Nullstellensatz makes $\fm$ maximal, and $\fp\subseteq\fm$ because $\fp\cR\subseteq s\cR\subseteq P$. Work near $\fm$ with compatible parameters as in Lemma~\ref{lem:rees}(4). Let $P_0$ be the vertex point over $\fm$, defined by $s=u_1=\cdots=u_k=0$ and the remaining parameters of $A_\fm$. Equivalently,
\[
  P_0=\left\{\sum_j f_jT^j\in\cR:f_0\in\fm\right\}.
\]
The differential construction of the invariant at the vertex gives
$\inv_{P_0}(I_w)=e:=\maxinv(I)$: the positive weights are those of the original center, now on the $u_i$, with $s$ of weight zero. This follows by transporting each step of the differential construction through $x_i=s^{w_i}u_i$; admissibility and the successive least derivative candidates are preserved. In this maximal center every homogeneous element $fT^j$, $f\in\cF_j$, $j\ge1$, belongs to the weighted ideal of level $1/d$.

Choose such an element $h=fT^j\notin P$. Spreading the vertex center and this membership to a principal neighborhood $D(G)$ of $P_0$, the local construction of the maximal center gives the following: at every maximal ideal $Q\in D(G)$ containing $I_w$, one has $\inv_Q(I_w)\preceq e$, and equality implies $h\in Q$. Indeed, in the equality case the spread center is the maximal center at $Q$ by uniqueness, and its positive-level weighted ideals lie in $Q$.

The grading acts by automorphisms $\sigma_\lambda(T)=\lambda T$, $\lambda\in\Q^\times$, and preserves $I_w$. Since $s\in P$, all negative-degree terms of $G$ vanish modulo $P$. The constant term of $\sigma_\lambda(G)$ modulo $P$ is nonzero, because $G\notin P_0$. Thus the image of $\sigma_\lambda(G)$ is a nonzero polynomial in $\lambda$ over the field $\cR/P$, and some $\lambda\in\Q^\times$ makes it nonzero. Then $Q=\sigma_\lambda^{-1}(P)$ lies in $D(G)$, contains $I_w$, and has the same invariant as $P$. But $\sigma_\lambda(h)=\lambda^jh\notin P$, so $h\notin Q$. Equality with $e$ is impossible, and $\inv_P(I_w)\prec e$ follows.
\end{proof}

\begin{lemma}\label{lem:derivations}
Let $\fm\supseteq I$ be a maximal ideal of $A$, and let $\delta$ be a $\Q$-derivation of $A$ with $\delta(I)\subseteq I$.
\begin{enumerate}
\item The derivation $\delta$ preserves the maximal center: $\delta(F_t(\fm))\subseteq F_t(\fm)$ for every $t$. Consequently, $\delta(\cF_j)\subseteq\cF_j$ in Notation~\ref{not:rees}.
\item Suppose that $\delta_1,\ldots,\delta_m$ are derivations of $A$ with $\delta_j(I)\subseteq I$, and that $\det(\delta_j(z_q))_{1\le j,q\le m}\notin\fm$ for some $z_1,\ldots,z_m\in A$. Then $\inv_\fm(I)$ has at most $\dim A-m$ nonzero entries.
\end{enumerate}
\end{lemma}

\begin{proof}
For (1), put $C=A_\fm$, write $e=\inv_\fm(I)$, and let $k$ be the number of its nonzero entries. Choose parameters $x=(x_1,\ldots,x_n)$ compatible with the maximal center and an integer $d\ge1$ such that every $de_i$ is integral. Choose an integer $D>d$. Put
\[
  B=A[\tau],\qquad \fn=\fm B+\tau B,\qquad
  I_D=IB+\tau^DB.
\]
Let $G_t\subseteq B_\fn$ be the weighted ideal for the ordered parameters
\[
  (x_1,\ldots,x_k,\tau,x_{k+1},\ldots,x_n)
\]
with weights $(e_1,\ldots,e_k,1/D,0,\ldots,0)$.

\medskip
\noindent\emph{Claim.} The filtration $(G_t)$ is the maximal center of $I_D$ at $\fn$.

\smallskip
\noindent\emph{Proof of the claim.}
Write $a_i=1/e_i$, so $a_i\le d<D$, and use the differential construction of \cite[Sections~3.2--3.3]{Brais}. After a prefix $a_1,\ldots,a_r$ has been chosen, the next exponent is the least candidate of the form
\[
  \frac{\sum_{j>r}\beta_j}
       {1-\sum_{i=1}^r\beta_i/a_i},
  \qquad 1-\sum_{i=1}^r\beta_i/a_i>0,
\]
for which the corresponding derivative ideal is the unit ideal; the variables include $\tau$. For $IB_\fn$, Theorem~\ref{thm:invariant}(1)--(2) shows that the maximal center is the smooth pullback of the original center, with $\tau$ of weight zero, so its positive exponents are $a_1,\ldots,a_k$. At every stage, the derivative ideal for $I_DB_\fn=IB_\fn+\tau^DB_\fn$ is the sum of the derivative ideals for the two summands. The derivative ideal contributed by $\tau^DB_\fn$ is proper unless $\beta_\tau\ge D$, and in that case the displayed candidate is at least $D$. Since the sum of two proper ideals in the local ring $B_\fn$ is proper, adjoining $\tau^D$ does not alter the first $k$ exponents, all of which are at most $d<D$. After those steps, $IB_\fn$ is admissible for the pulled-back center, while $\tau^D$ contributes the next exponent $D$ and the parameter $\tau$, obtained by differentiating $D-1$ times. The enlarged center is admissible for $I_DB_\fn$, so the construction stops. Uniqueness of the maximal center identifies its filtration with $(G_t)$, proving the claim.
\medskip

Extend $\delta$ to $C$ and set $\delta(\tau)=0$. Under the identification
\[
  B_\fn/\tau^DB_\fn\cong C[\tau]/(\tau^D),
\]
the formulas
\[
  \Phi(\tau)=\tau,\qquad
  \Phi(f)=\sum_{j=0}^{D-1}\frac{\tau^j\delta^j(f)}{j!}\quad(f\in C)
\]
define an automorphism; its inverse is obtained by replacing $\tau$ by $-\tau$. Since $\delta(IC)\subseteq IC$, the automorphism $\Phi$ preserves the image of $I_D$. Put
\[
  x'_i=\sum_{j=0}^{D-1}\frac{\tau^j\delta^j(x_i)}{j!}.
\]
Modulo $(\fn B_\fn)^2$, each $x'_i$ differs from $x_i$ by a scalar multiple of $\tau$. Hence
\[
  (x'_1,\ldots,x'_k,\tau,x'_{k+1},\ldots,x'_n)
\]
is a regular system of parameters. Let $G'_t$ be the weighted ideals defined by these parameters and the same weights as $G_t$. Modulo $\tau^D$, the automorphism $\Phi$ preserves the image of $I_D$ and carries the image of $G_1$ to that of $G'_1$. Therefore
\[
  I_DB_\fn\subseteq G'_1+\tau^DB_\fn=G'_1,
\]
where the equality follows from $\tau^D\in G'_1$. Thus the new center is admissible for $I_D$. The claim identifies its weight vector as maximal, so uniqueness gives $G'_t=G_t$ for every $t$.

In particular, $x'_i\in G_{e_i}$ for $1\le i\le k$. Under the natural map $B_\fn\to C[[\tau]]$, expand an element of $G_{e_i}$ in the monomial generators of that ideal. Only generators of $\tau$-degree zero or one can contribute to the coefficient of $\tau$. In the first case the $x$-weight is at least $e_i$, and in the second it is at least $e_i-1/D$. Hence the coefficient belongs to $F_{e_i-1/D}(\fm)$, and therefore
\[
  \delta(x_i)\in F_{e_i-1/D}(\fm)=F_{e_i}(\fm).
\]
The equality holds because $e_i$ and every $x$-monomial weight lie in $\frac1d\Z$, whereas $1/D<1/d$. The Leibniz rule now gives $\delta(F_t(\fm))\subseteq F_t(\fm)$ for every $t$. Applying this at every maximal ideal containing the chosen component and using the definition in Theorem~\ref{thm:invariant}(4) proves $\delta(\cF_j)\subseteq\cF_j$.

For (2), let $k$ be the number of nonzero entries of $e=\inv_\fm(I)$. For $0<\varepsilon\le e_k$ we have $F_\varepsilon(\fm)=(x_1,\ldots,x_k)$, so by (1) every $\delta_j(x_i)$ with $i\le k$ lies in $\fm$. Writing $\delta_j=\sum_i\delta_j(x_i)\partial_i$, we see that the matrix $(\delta_j(z_l))$ is the product of $(\delta_j(x_i))_{j,i}$ and $(\partial_iz_l)_{i,l}$. Modulo $\fm$, the first factor has zero columns for $i\le k$, so the product has rank at most $n-k$ over $\kappa(\fm)$. Hence $m\le n-k$.
\end{proof}

\section{Principalizing ideals while preserving the obstruction}
\label{sec:principalization}

Throughout this section, $A$ is a smooth finitely generated factorial $\Q$-domain, and $I\subsetneq A$ is a nonzero ideal. We fix a component $V(\fp)$ of the maximal locus of $I$ and use Notation~\ref{not:rees}; in particular, $k\ge1$ is the number of nonzero entries of $e=\maxinv(I)$. Since $A$ is a UFD and $\fp\ne0$, the prime ideal $\fp$ contains a prime element $\pi$. Let $c(I)$ denote the number of components of the maximal locus.

One step replaces $(A,I)$ by a pair $(B,I_1)$, where $B=A$ or $B$ is a smooth finitely generated factorial $\Q$-domain containing $A$, and factors $IB$ as $IB=\alpha I_1$ for some $\alpha\in B$. The complexity decreases in the following sense: either $I_1=B$, or $\maxinv(I_1)\prec\maxinv(I)$, or $\maxinv(I_1)=\maxinv(I)$ and $c(I_1)<c(I)$. We distinguish the divisorial case $k=1$ from the case $k\ge2$.

\subsection{The divisorial step}

\begin{lemma}\label{lem:divisorial}
Assume $k=1$, so $e=(1/a,0,0,\ldots)$ for a positive integer $a$. Then $\fp=\pi A$, $I\subseteq\pi^aA$, and for $I_1=(I:\pi^a)$ one has $I=\pi^aI_1$. Moreover, either $I_1=A$, or $\maxinv(I_1)\prec\maxinv(I)$, or $\maxinv(I_1)=\maxinv(I)$ and $c(I_1)<c(I)$.
\end{lemma}

\begin{proof}
At every maximal ideal containing $\fp$, the ideal $\fp$ is generated by one parameter. Thus $\fp$ is a height-one prime, hence $\fp=\pi A$ because $A$ is factorial. Admissibility gives $I\subseteq\pi^aA$, and therefore $I=\pi^aI_1$.

No maximal ideal can contain both $I_1$ and $\pi$: otherwise $IA_\fm\subseteq\fm^{a+1}A_\fm$, so the first weight of $\inv_\fm(I)$ would be at most $1/(a+1)$, contradicting $\inv_\fm(I)=e$. Hence $\pi$ is a unit at every maximal ideal containing $I_1$, where $I_1$ and $I$ have the same localization and the same invariant. Thus the chosen component $V(\fp)$ disappears and no new maximal component is created. The conclusion follows from Theorem~\ref{thm:invariant}(3).
\end{proof}

\subsection{The torsor step}
Assume $k\ge2$. Let
\[
  \mathbf h=(h_0,h_1,\ldots,h_\ell)=\bigl(\pi T,\ gT^j:\ g\in G_j,\ 1\le j\le d\bigr)
\]
be the tuple of elements of $\cR_+$ obtained by listing $\pi T$ followed by the elements $gT^j$ with $g\in G_j$ and $1\le j\le d$. Write $h_i=g_iT^{j_i}$, where $1\le j_i\le d$ and $g_i\in\cF_{j_i}$; thus $g_0=\pi$ and $j_0=1$. By Lemma~\ref{lem:rees}(2), the entries of $\mathbf h$ generate $\cR_+$. Put
\[
  U=J_\cR(\mathbf h)=\cR[\sigma_0,\ldots,\sigma_\ell]/\Bigl(\sum_{i=0}^\ell h_i\sigma_i-1\Bigr),
\]
so that $\Spec U\to\Spec \cR\setminus V(\cR_+)$ is the Jouanolou affine-space bundle associated with $\mathbf h$, and let $I_1=I_wU$.

\begin{lemma}\label{lem:torsor}
The ring $U$ is a smooth finitely generated factorial $\Q$-domain, the homomorphism $A\to U$ is injective, and $IU=s^dI_1$.
\end{lemma}

\begin{proof}
By Lemmas~\ref{lem:rees} and~\ref{lem:jouanolou}, the ring $U$ is a smooth finitely generated $\Q$-domain, $\cR\to U$ is injective, and $IU=s^dI_1$. Modulo $s$ one has
\[
  U/sU\cong J_{\cR/s\cR}(\overline{\mathbf h}).
\]
The base ring is a domain by Lemma~\ref{lem:rees}(5), and the row $\overline{\mathbf h}$ is nonzero: after localizing at a maximal ideal over $\fp$ and writing $w_1=de_1$, one has $x_1\in\cF_{w_1}A_\fm\setminus\cF_{w_1+1}A_\fm$, so some generator in $G_{w_1}$ has nonzero initial form. Lemma~\ref{lem:jouanolou}(3) therefore shows that $U/sU$ is a domain, so $s$ is prime in $U$.

After inverting $s$,
\[
  U[1/s]\cong J_{A[T,T^{-1}]}(\mathbf h).
\]
The Laurent polynomial ring is factorial, the row contains the prime element $\pi T$, and it also contains an element $gT$ not divisible by $\pi T$: otherwise the height-$k$ prime $\fp$, with $k\ge2$, would be generated by $\pi$. Hence Lemma~\ref{lem:jouanolou}(5) makes $U[1/s]$ factorial, and Nagata's criterion makes $U$ factorial. The inclusion $A\hookrightarrow U$ follows from $A\subseteq\cR\hookrightarrow U$.
\end{proof}

\begin{lemma}\label{lem:torsor-invariant}
The ideal $I_1$ is nonzero. Moreover, either $I_1=U$, or $\maxinv(I_1)\prec\maxinv(I)$, or $\maxinv(I_1)=\maxinv(I)$ and $c(I_1)<c(I)$.
\end{lemma}

\begin{proof}
The ideal $I_1$ contains $fT^d\ne0$ for $0\ne f\in I$, since $\cR\to U$ is injective. Let $Q\supseteq I_1$ be maximal and put $P=Q\cap\cR$. The contractions $P$ and $Q\cap A$ are maximal by the weak Nullstellensatz. Since the $h_i$ generate the unit ideal in $U$, some $h_i\notin P$; hence $U_Q$ is a localization of a polynomial ring over $\cR_P$, and
\[
  \inv_Q(I_1)=\inv_P(I_w).
\]
If $s\in Q$, Theorem~\ref{thm:drop} gives a strict decrease. If $s\notin Q$, then $\cR_P$ is a localization of $A_{Q\cap A}[T,T^{-1}]$ by Lemma~\ref{lem:rees}(3), so Theorem~\ref{thm:invariant}(2) gives
\[
  \inv_Q(I_1)=\inv_{Q\cap A}(I).
\]
Moreover $Q\cap A$ cannot contain $\fp$, for then $s^{j_i}h_i=g_i\in\fp\subseteq P$ would imply $h_i\in P$ for every $i$.

It follows that every invariant of $I_1$ is at most $\maxinv(I)$, and equality can occur only over a component $V(\fp')$ of the old maximal locus with $\fp'\ne\fp$. Conversely, let $Q$ be any maximal ideal containing $\fp'U$. Then $s\notin Q$: otherwise $\fp U\subseteq sU\subseteq Q$, contradicting $(\fp+\fp')U=U$. Since $I\subseteq\fp'$, the identity $IU=s^dI_1$ gives $I_1U_Q=IU_Q\subseteq QU_Q$, so $I_1\subseteq Q$ and the invariant at $Q$ equals $\maxinv(I)$. Finally,
\[
  \cR/\fp'\cR\cong(A/\fp')[T,T^{-1}],
\]
and some generator from $G_1$ survives modulo $\fp'$, so Lemma~\ref{lem:jouanolou}(3) shows that $U/\fp'U$ is a nonzero domain. Thus the ideals $\fp'U$, for the old components other than $\fp$, are proper pairwise-comaximal primes and are exactly the components that can remain. The claimed decrease follows.
\end{proof}

\subsection{Termination}
The torsor step raises dimension, so we keep track of the new smooth directions. Fix $N\ge0$. Say that a triple $(B,\fb,m)$ satisfies $(\star_N)$ if
\[
  \dim B\le N+m
\]
and, at every maximal ideal $\fm\supseteq\fb$, there are $m$ derivations preserving $\fb$ whose values on $m$ elements have determinant outside $\fm$. By Lemma~\ref{lem:derivations}(2), every invariant of such an ideal has at most $N$ nonzero entries. The initial triple $(A,I,0)$ satisfies $(\star_{\dim A})$.

\begin{lemma}\label{lem:length-control}
Suppose that $(A,I,m)$ satisfies $(\star_N)$. In the divisorial step, $(A,I_1,m)$ satisfies $(\star_N)$. In the torsor step, $(U,I_1,m+\ell+1)$ satisfies $(\star_N)$.
\end{lemma}

\begin{proof}
In the divisorial step, the dimension bound is unchanged and $V(I_1)\subseteq V(I)$. By Lemma~\ref{lem:derivations}(1), a derivation preserving $I$ satisfies $\delta(\pi)=b\pi$ and therefore preserves $(I:\pi^a)$, since
\[
  \pi^a\delta(f)=\delta(\pi^af)-ab\,\pi^af\in I
  \qquad(\pi^af\in I).
\]

In the torsor step, $\dim U=\dim A+1+\ell\le N+m+1+\ell$. Let $Q\supseteq I_1$ be maximal. Then $Q\cap A$ is maximal and contains $I$, so choose the $m$ derivations and elements supplied by $(\star_N)$ there, and choose $i_0$ with $h_{i_0}\notin Q$. Lemma~\ref{lem:derivations}(1) lets each old derivation act coefficientwise on $\cR$ and extend to $U$ by
\[
  \delta(\sigma_i)=-\sigma_i\sum_r\sigma_r\delta(h_r).
\]
Together with the Euler derivation
\[
  E(fT^q)=qfT^q,\qquad E(\sigma_i)=-j_i\sigma_i
\]
and the $\ell$ vertical derivations
\[
  h_r\frac{\partial}{\partial\sigma_{i_0}}
  -h_{i_0}\frac{\partial}{\partial\sigma_r}
  \qquad(r\ne i_0),
\]
these give $m+\ell+1$ derivations preserving $I_1$. On the elements $z_1,\ldots,z_m,h_{i_0}$ and $\sigma_r$ for $r\ne i_0$, their value matrix is block upper triangular with determinant
\[
  \det(\delta_j(z_q))\,j_{i_0}h_{i_0}(-h_{i_0})^\ell\notin Q.
\]
Thus $(U,I_1,m+\ell+1)$ satisfies $(\star_N)$.
\end{proof}

\subsection{Real points}\label{subsec:principalization-real-points}

\begin{lemma}\label{lem:sphere-bundle}
In the torsor step, let $K\subset\Spec(A)(\R)$ be compact, put $E_i=2\,d!/j_i$ (an even integer), and let $K'\subset\Spec(U)(\R)$ be the set of real points $z'$ such that
\[
  \begin{gathered}
    z'|_A\in K,\qquad z'(s)\ge0,\qquad
    \sum_{i=0}^\ell z'(h_i)^{E_i}=1,\\
    z'(\sigma_i)=z'(h_i)^{E_i-1}\qquad(0\le i\le\ell).
  \end{gathered}
\]
Then $K'$ is compact, and $K'\to K$ is a monotone surjection.
\end{lemma}

\begin{proof}
The last two conditions are compatible with $\sum_i h_i\sigma_i=1$. If $z'\in K'$ and $z=z'|_A$, then $s^{j_i}h_i=g_i$ and $j_iE_i=2d!$ give
\begin{equation}\label{eq:s-norm}
  z'(s)^{2d!}=\sum_{i=0}^\ell z(g_i)^{E_i}=: \nu(z).
\end{equation}
The defining conditions are closed, $|z'(h_i)|,|z'(\sigma_i)|\le1$, and \eqref{eq:s-norm} bounds $z'(s)$ on $K'$. Since $U$ is generated over $A$ by $s$, the $h_i$, and the $\sigma_i$, while elements of $A$ are bounded on the compact set $K$, the criterion of Subsection~\ref{subsec:real-points} shows that $K'$ is compact. The restriction map $K'\to K$ is continuous.

The function $\nu(z)$ is positive exactly when $\fp\nsubseteq\ker z$: every $g_i$ lies in $\fp$, while the coefficients with $j_i=1$ generate $\fp$. If $\nu(z)>0$, equation \eqref{eq:s-norm} forces $z'(s)=\rho:=\nu(z)^{1/(2d!)}>0$. There is then a unique point over $z$: extend $z$ to $A[T,T^{-1}]$ by $z'(T)=\rho^{-1}$ and put $z'(\sigma_i)=z'(h_i)^{E_i-1}$.

Now suppose $\fp\subseteq\ker z$. Then every point over $z$ sends $s$ to zero. Choose a maximal ideal $\fm\supseteq\ker z$ and use the local presentation in Lemma~\ref{lem:rees}(4). Points of $\cR$ over $z$ with $s=0$ are parametrized by $a=(a_1,\ldots,a_k)\in\R^k$, where $a_r$ is the value of $x_rT^{w_r}$. Write $H_i(a)$ for the value of $h_i$ at this point. Each $H_i$ is weighted homogeneous of degree $j_i$, so, with
\[
  \lambda\star a=(\lambda^{w_1}a_1,\ldots,\lambda^{w_k}a_k),
  \qquad \varphi(a)=\sum_i H_i(a)^{E_i},
\]
one has $\varphi(\lambda\star a)=\lambda^{2d!}\varphi(a)$. Moreover, if $\varphi(a)=0$, then the even exponents force every $H_i(a)$ to vanish. The corresponding point therefore vanishes on the ideal $\cR_+$ generated by the $h_i$, including every $x_rT^{w_r}$, and hence $a=0$.

The fiber of $K'\to K$ over $z$ is therefore the level set $\varphi(a)=1$, with the values of the $\sigma_i$ determined by the defining equations. The normalization
\[
  a\longmapsto \varphi(a)^{-1/(2d!)}\star a
\]
is a continuous surjection from $\R^k\setminus\{0\}$ onto this fiber. Since $k\ge2$, the source is connected. Thus every fiber is nonempty and connected, and $K'\to K$ is a monotone surjection.
\end{proof}

\subsection{The principalization result}

\begin{proposition}[Principalization extension]\label{prop:principalization-extension}
Let $A$ be a smooth finitely generated factorial $\Q$-domain, let $a,b\in A$ be not both zero, and let $K\subset\Spec(A)(\R)$ be compact. There are a smooth finitely generated factorial $\Q$-domain $A'$, an injection $A\hookrightarrow A'$, and a compact set $K'\subset\Spec(A')(\R)$ such that $(a,b)A'$ is principal and the induced map $K'\to K$ is a monotone surjection.
\end{proposition}

\begin{proof}
Put $I=(a,b)$ and $N=\dim A$. We prove the following auxiliary assertion. Let $B$ be a smooth finitely generated factorial $\Q$-domain, let $\fb\subseteq B$ be a nonzero ideal, let $m\ge0$ be such that $(B,\fb,m)$ satisfies $(\star_N)$, and let $K\subseteq\Spec(B)(\R)$ be compact. Then there are an injection $B\hookrightarrow B'$ into a smooth finitely generated factorial $\Q$-domain and a compact set $K'\subseteq\Spec(B')(\R)$ such that $\fb B'$ is principal and $K'\to K$ is a monotone surjection. The proposition is the initial case $(B,\fb,m)=(A,I,0)$.

If $\fb=B$, take $B'=B$ and $K'=K$. For $\fb\ne B$ we argue by well-founded induction on the pair $(\maxinv(\fb),c(\fb))$, ordered lexicographically, with $\preceq$ on the first entry and the usual order on $\N$ on the second. This is well-founded by Theorem~\ref{thm:invariant}(5), because all invariants of ideals in triples satisfying $(\star_N)$ have at most $N$ nonzero entries. Let $k$ be the number of nonzero entries of $\maxinv(\fb)$, choose a component of the maximal locus of $\fb$, and apply the divisorial step if $k=1$ or the torsor step if $k\ge2$. This gives a smooth finitely generated factorial $\Q$-domain $B_1$ with an injection $B\hookrightarrow B_1$ (namely $B_1=B$, or $B_1=U$ by Lemma~\ref{lem:torsor}), a compact set $K_1\subset\Spec(B_1)(\R)$ with $K_1\to K$ a monotone surjection (with $K_1=K$ in the divisorial case and $K_1$ as in Lemma~\ref{lem:sphere-bundle} in the torsor case), an ideal $\fb_1\subseteq B_1$, and an element $\alpha\in B_1$ with $\fb B_1=\alpha\fb_1$ (where $\alpha=\pi^a$ or $\alpha=s^d$). Let $m_1=m$ in the divisorial step and $m_1=m+\ell+1$ in the torsor step. By Lemmas~\ref{lem:divisorial}, \ref{lem:torsor-invariant} and~\ref{lem:length-control}, the triple $(B_1,\fb_1,m_1)$ satisfies $(\star_N)$, and either $\fb_1=B_1$ or $(\maxinv(\fb_1),c(\fb_1))$ is strictly smaller than $(\maxinv(\fb),c(\fb))$ in this lexicographic order.

In either case, the assertion for $(B_1,\fb_1,m_1)$ and $K_1$ holds, trivially if $\fb_1=B_1$ and by induction otherwise. It gives $B_1\hookrightarrow B'$ and a compact set $K'\subset\Spec(B')(\R)$ with $\fb_1B'=\alpha'B'$ and $K'\to K_1$ a monotone surjection. Then $\fb B'=\alpha\alpha'B'$ is principal, $B\to B'$ is injective, and $K'\to K$ is a monotone surjection by Lemma~\ref{lem:monotone-composition}.
\end{proof}

\section{Construction of the B\'ezout domain}\label{sec:construction}

\begin{construction}\label{construction:ring}
Fix a bijection $\langle\cdot,\cdot\rangle:\N\times\N\to\N$ with $i\le\langle i,j\rangle$ for all $i,j$, for instance the Cantor pairing. Starting with $A_0=\Q[x,y]$ and the circle $K_0$, we construct injections
\[
  A_0\hookrightarrow A_1\hookrightarrow A_2\hookrightarrow\cdots
\]
of smooth finitely generated factorial $\Q$-domains and compact sets $K_n\subset\Spec(A_n)(\R)$ such that each induced map $K_{n+1}\to K_n$ is a monotone surjection.

Every finitely generated $\Q$-algebra is countable. Once $A_i$ has been constructed, choose a surjection
\[
  \eta_i:\N\longrightarrow A_i\times A_i.
\]
To pass from $A_n$ to $A_{n+1}$, write $n=\langle i,j\rangle$. Then $i\le n$, and we let $(a,b)$ be the image of the pair $\eta_i(j)$ in $A_n$. If $a=0$ or $b=0$, the ideal $(a,b)A_n$ is principal, and we put $A_{n+1}=A_n$ and $K_{n+1}=K_n$. Otherwise, write $a=ca'$ and $b=cb'$ with $a'$ and $b'$ coprime, which is possible because $A_n$ is a UFD. Apply Proposition~\ref{prop:principalization-extension} to $a'$, $b'$, and the compact set $K_n$, and let $A_{n+1}$ and $K_{n+1}$ be the resulting ring and compact set. Then $(a,b)A_{n+1}=c\,(a',b')A_{n+1}$ is principal.

Identify every $A_n$ with its image in later rings and set
\begin{equation}\label{eq:union-ring}
  R=\bigcup_{n\ge0}A_n.
\end{equation}
\end{construction}

\begin{proposition}\label{prop:bezout-domain}
The ring $R$ is a B\'ezout domain.
\end{proposition}

\begin{proof}
Since all transition maps are injective, the increasing union $R$ is a domain. Let $a,b\in R$. Choose $i$ such that $a,b\in A_i$, and choose $j\in\N$ with $\eta_i(j)=(a,b)$. This pair is processed while passing from $A_n$ to $A_{n+1}$, where $n=\langle i,j\rangle$. Hence
\[
  (a,b)A_{n+1}=gA_{n+1}
\]
for some $g\in A_{n+1}$. Extending to $R$ gives $(a,b)R=gR$. Thus $R$ is a B\'ezout domain.
\end{proof}

The functions $x$ and $y$ on $K_n$ are pulled back from $K_0$, so the line bundle
\[
  L_n=L_{x,y}\quad\text{on }K_n
\]
is the pullback of $L_0$ along the composite of the transition maps. By Lemmas~\ref{lem:orientation-descent} and~\ref{lem:monotone-composition},
\begin{equation}\label{eq:nonorientable-stages}
  L_n\text{ is nonorientable for every }n\ge0.
\end{equation}

\begin{lemma}\label{lem:key-obstruction}
There are no elements $p,q,u,v,\mu\in R$ such that
\[
  \mu\,(p,q)\,M\binom uv=1.
\]
\end{lemma}

\begin{proof}
Suppose such elements exist, and choose $N$ such that all five lie in $A_N$. The identity already holds in $A_N$, because $A_N\hookrightarrow R$ is injective. For $z\in K_N$, put
\[
  \xi(z)=M(z)\binom{u(z)}{v(z)}\in L_N(z).
\]
The assumed identity gives $\mu(z)(p(z),q(z))\,\xi(z)=1$, so $\xi(z)\ne0$. Thus $\xi$ is a continuous nowhere-zero section of $L_N$, contradicting \eqref{eq:nonorientable-stages}.
\end{proof}

\begin{proof}[Proof of Theorem~\ref{thm:main}]
Take the ring $R$ from Construction~\ref{construction:ring}. By Proposition~\ref{prop:bezout-domain}, it is a B\'ezout domain. It remains to show that it is not an elementary divisor domain.

Suppose, to the contrary, that the matrix $M$ has a Smith normal form over $R$:
\begin{equation}\label{eq:putative-smith-form}
  PMQ=\diag(d_1,d_2),\qquad
  P,Q\in\GL_2(R),\qquad d_1\mid d_2.
\end{equation}
Multiplication by invertible matrices preserves the ideal generated by all matrix entries. By \eqref{eq:matrix-identities}, the entries of $M$ generate $R$. Hence $(d_1,d_2)=R$. Since $d_1\mid d_2$, this ideal equals $(d_1)$, so $d_1$ is a unit.

Let $(p,q)$ be the first row of $P$ and $(u,v)^{\mathsf T}$ the first column of $Q$. Comparing upper-left entries in \eqref{eq:putative-smith-form} gives
\[
  (p,q)M(u,v)^{\mathsf T}=d_1.
\]
Multiplying by $d_1^{-1}$ contradicts Lemma~\ref{lem:key-obstruction}. Thus $M$ has no Smith normal form over $R$, so $R$ is not an elementary divisor domain.
\end{proof}

\end{document}